\documentclass[review,12pt]{elsarticle}

\usepackage{amssymb,amsmath,amsthm,graphicx}
\usepackage{enumerate}
\usepackage{float}
\usepackage{stmaryrd,mathbbol,mathrsfs} 
\usepackage{multicol}
\usepackage{float}
\usepackage{stackrel}
\usepackage{color}
\usepackage{tikz}
\usetikzlibrary{matrix,arrows}
\usepackage{capt-of}
\usepackage{array}

\newtheorem{theorem}{Theorem}[section]
\newtheorem*{theorem*}{Theorem}
\newtheorem{lemma}[theorem]{Lemma}

\newtheorem{corollary}[theorem]{Corollary}

\theoremstyle{remark}

\theoremstyle{definition}
\newtheorem{example}[theorem]{Example}

\makeatletter
\newcommand{\Spvek}[2][r]{%
  \gdef\@VORNE{1}
  \left(\hskip-\arraycolsep%
    \begin{array}{#1}\vekSp@lten{#2}\end{array}%
  \hskip-\arraycolsep\right)}

\def\vekSp@lten#1{\xvekSp@lten#1;vekL@stLine;}
\def\vekL@stLine{vekL@stLine}
\def\xvekSp@lten#1;{\def\temp{#1}%
  \ifx\temp\vekL@stLine
  \else
    \ifnum\@VORNE=1\gdef\@VORNE{0}
    \else\@arraycr\fi%
    #1%
    \expandafter\xvekSp@lten
  \fi}
\makeatother
\journal{}

\begin{document}

\begin{frontmatter}



\title{On generating of idempotent aggregation functions on finite lattices
\footnote{Preprint of an article published by Elsevier in the Information Sciences 430–431, Pages 39-45. It is available online at: \newline https://www.sciencedirect.com/science/article/pii/S0020025516316164}}

\author[up]{Michal Botur}
\ead{ michal.botur@upol.cz}

\author[up]{Radom\'ir Hala\v{s}}
\ead{radomir.halas@upol.cz}

\author[up,stu]{Radko Mesiar}
\ead{radko.mesiar@stuba.sk}

\author[up,sav]{Jozef P\'ocs}
\ead{pocs@saske.sk}

\address[up]{Department of Algebra and Geometry, Faculty of Science, Palack\'y University Olomouc, 17. listopadu 12, 771 46 Olomouc, Czech Republic}
\address[stu]{Department of Mathematics and Descriptive Geometry, Faculty of Civil Engineering, Slovak University of Technology in Bratislava, Radlinsk\'eho 11, 810 05 Bratislava 1, Slovakia}
\address[sav]{Mathematical Institute, Slovak Academy of Sciences, Gre\v s\'akova 6, 040 01 Ko\v sice, Slovakia}

\begin{abstract}
In a recent paper \cite{HP} we proposed the study of aggregation functions on lattices via clone theory approach. Observing that aggregation functions on lattices just correspond to $0,1$-monotone clones, as the main result of \cite{HP}, we have shown that all aggregation functions on a finite lattice $L$ can be obtained as usual composition of lattice operations $\wedge,\vee$, and certain unary and binary aggregation functions.

The aim of this paper is to present a generating set for the class of intermediate (or, equivalently, idempotent) aggregation functions. This set consists of lattice operations and certain ternary idempotent aggregation functions.
\end{abstract}

\begin{keyword}
Idempotent aggregation function\sep generating set \sep finite lattice.

\MSC 06B99 

\end{keyword}

\end{frontmatter}

\section{Introduction}

Aggregation is a process when (usually numerical) data are merged in a single output. Mathematically, the process of aggregation is based on the concept of aggregation function. 

The most natural examples of aggregation functions widely used in experimental sciences are means and averages (such as e.g. the arithmetic mean). These belong to a widely studied class of so-called internal aggregation functions, firstly mentioned by Cauchy already in 1821, with a huge variety of applications.    
Nowadays, aggregation functions are successfully applied in many different branches of science, we can mention e.g. social sciences, computer science, psychology etc. 

As the process of aggregation somehow "synthesizes" the input data, aggregation functions cannot be arbitrary and have to fulfill some natural minimal requirements. This can be translated  into the condition that the output value should lie in the same interval as the input ones, and the least and the greatest values should be preserved.
Another widely accepted property of aggregation functions is that the output value should increase or at least stay constant whenever the input values increase.

The theory of aggregation functions is well developed in a case when the input (and, consequently, the output) values of 
these functions lie in a nonempty interval of reals, bounded or not. For details, we can refer the reader to the comprehensive monograph \cite{GMRP}.
During several last years, the theory was enlarged to lattice-based data, i.e. when 
the input (as well as the output) data have a structure of a lattice. This more general approach 
allows us to work e.g. with data which are not linearly ordered or when information 
about the input data is incomplete.
   
Recall that a lattice is 
an algebra $(L;\vee,\wedge)$, where $L$ is a nonempty set with two binary operations $\vee$ and $\wedge$ representing 
suprema and infima. Let us mention that lattice theory is a well established discipline of general algebra. There are 
several monographs on this topic, among them the most frequently used and quoted is the book by G. Gr\"atzer, \cite{G1}.
  
The formal definition of aggregation function is as follows:\\
\noindent
An aggregation function on a bounded lattice $L$ is a function $A: L^n\to L$ that 
\begin{itemize}
\item [(i)] is nondecreasing (in each variable), i.e. for any $\mathbf{x}, \mathbf{y}\in L^n$: 
$$A(\mathbf{x})\leq A(\mathbf{y})\,\, \text{ whenever }\,\, \mathbf{x}\leq \mathbf{y},$$

\item [(ii)] fulfills the boundary conditions
$$A(0,\dots,0)=0 \quad\text{ and }\quad A(1,\dots,1)=1.$$
\end{itemize}

The integer $n$ represents the arity of the aggregation function.

The set of all $n$-ary aggregation functions can be naturally ordered via component-wise ordering, i.e., if $A$ and $B$ are $n$-ary aggregation functions, then $A\leq B$ provided $A(\mathbf{x})\leq B(\mathbf{x})$ for all $\mathbf{x}\in L^n$. Note, that with respect to such ordering the set of all $n$-ary aggregation functions forms a bounded lattice. In the sequel, when referring to the order of aggregation functions, we have in mind this component-wise ordering. 

Clearly one of the central problems in aggregation theory is to give their construction methods. Much of work has already been done in this direction which fact can be easily demonstrated by the extensive literature including book chapters, we refer the reader e.g. a standard monograph \cite{GMRP}. It can be recognized that methods like composed aggregation, weighted aggregation, forming ordinal sums etc., look quite different and each of them relies on a very specific approach. In a classical case, the idea is based on standard arithmetical operations on the real line and fixed real functions. 

From the point of view of universal algebra, aggregation functions on a lattice $L$ form a clone (or, equivalently, a composition-closed set of functions containing the projections), the so-called aggregation clone of $L$. Recall that the clone theory is a very well established discipline of universal algebra. In principle, it deals with function algebras and its development was mostly initiated by studies in many-valued logic. For details, we refer the reader to the books \cite{Lau,KP} or to the overview paper \cite{KPS}.    

The clone of aggregation functions on $L$ forms a subclone of the so-called monotone clone of $L$ consisting of all functions on $L$ preserving the lattice order. It is well-known that for a finite lattice $L$, any clone on $L$ containing the lattice operations is finitely generated. Although generating sets of monotone clones on finite lattices can be found e.g. in \cite{Lau}, these cannot be directly used for generating of any of its subclones. In \cite{HP}, for any $n$-element lattice, we presented explicitly at most $(2n+2)$-element generating set of aggregation functions, consisting of lattice binary operations, at most $n$ unary functions and at most $n$ binary aggregation functions. Consequently, we have shown that any aggregation function on $L$ arises as the usual term composition of the above mentioned set of generating functions. Moreover, contrary to the case of the monotone clone, we have shown that each generating set of the aggregation clone containing the lattice operations, must posses at least one another operation of arity higher than one. Let us mention that for the classical case when $L=[0,1]$ is the unit interval of reals, another generating set of functions has been presented in \cite{HMP3}. 


Although the above mentioned generating set allows us to construct the set of all aggregation functions on a bounded lattice $L$, there is a question how to generate some of its important subclasses being composition-closed. For example, all intermediate (or, equivalently, idempotent) aggregation functions form such a class. 
Since there is only one unary idempotent function, namely the identity function, the generating set described in \cite{HP} cannot be applied.
The aim of this paper is to present a generating set for the class of all idempotent aggregation functions, members of which will be itself idempotent aggregation functions. In fact, as the main result we will show that certain ternary idempotent aggregation functions, together with the lattice operations form a generating set. We expect the extension of our results for certain important composition-closed subclasses of idempotent aggregation functions, especially the class of Sugeno integrals \cite{HMP1,HMP2} on bounded distributive lattices.

\section{Idempotent aggregation functions}

For a set $X$ and a natural number $n$, denote by $O_{n}(X)$ the set of all $n$-ary functions on $X$, i.e. the 
mappings $f\colon X^n\to X$. We put $O(X)=\bigcup_{n=1}^{\infty} O_{n}(X)$. 

Let $X$ be a set and $n\in\mathbb N$ be a positive integer. For any $i\leq n$, the {\it $i$-th $n$-ary projection} 
is for all $x_1,\dots,x_n\in X$ defined by 
\begin{equation}\label{eqr1}
p_i^n(x_1,\dots,x_n):=x_i.
\end{equation}
Composition of functions on $X$ forms from one $k$-ary function $f\colon X^k\to X$ and $k$ $n$-ary functions $g_1,\dots,g_k\colon X^n\to X$, 
an $n$-ary function $f(g_1,\dots,g_k)\colon X^n\to X$ defined by
\begin{equation}\label{eqr2}
f\big(g_1,\dots,g_k\big)(x_1,\dots,x_n):=f\big(g_1(x_1,\dots,x_n),\dots,g_k(x_1,\dots,x_n)\big),
\end{equation}
for all $x_1,\dots,x_n\in X$. 

Notice that for $k=n=1$, composition is a usual product of selfmaps.

Given a subset $T\subseteq O(X)$ of functions on $X$, we denote by $[T]$ the least set (with respect to set inclusion) of functions on $X$ containing $T$ and being closed under composition. $T$ is called composition-closed whenever $T=[T]$.   
Obviously, $O(X)$ represents the composition-closed set. As composition of $n$-ary functions on $X$ yields an $n$-ary function, the set $O_n(X)$ is composition-closed, too.  

Given a bounded lattice $L$, recall
that an $n$-ary aggregation function $f$ on $L$ is said to be idempotent if it satisfies  
$ f(x,\dots,x)=x$ for all $x\in L$.

For a positive integer $n\geq 1$ denote by $\mathsf{Agg}^n(L)$ the set of all $n$-ary aggregation functions defined on the lattice $L$. Further, we put $\mathsf{Agg}(L)=\bigcup_{i=1}^{\infty} \mathsf{Agg}^n(L)$, i.e., $\mathsf{Agg}(L)$ denotes the set of all aggregation functions on $L$. Analogously, we denote by $\mathsf{Id}^n(L)$ the set of all $n$-ary idempotent aggregation functions on $L$ and by $\mathsf{Id}(L)$ the set of all idempotent aggregation functions. 
For any subset $T\subseteq \mathsf{Agg}(L)$ of aggregation functions, the symbol $T^n$ denotes all $n$-ary functions from $T$, i.e., $T^n= T\cap \mathsf{Agg}^n(L)$. 

Considering the functions $\bigwedge_{i=1}^n x_i$ and $\bigvee_{i=1}^n x_i$ as dominating or
dominated functions on a lattice $L$, we obtain three main classes of aggregation functions:
conjunctive functions, disjunctive functions and intermediate functions.

Recall that an aggregation function $f:L^n\to L$ on $L$ is intermediate if it fulfills the inequalities 
\begin{equation}\label{eq1}
 \bigwedge_{i=1}^n x_i \leq f(x_1,\dots,x_n)\leq \bigvee_{i=1}^n x_i 
\end{equation}
for all $(x_1,\dots,x_n)\in L^n$. 

One can easily prove the following elementary but for our purposes important property of idempotent functions: 
\begin{lemma}
Any projection is idempotent and idempotent aggregation functions form a composition-closed class.
\end{lemma}

Recall that a composition-closed set of functions containing all the projections \eqref{eqr1} is referred to as clone. Hence for any bounded lattice $L$, the set $\mathsf{Id}(L)$ forms a clone, in the sequel called the idempotent clone on $L$.

Another important clone, more precisely a subclone of the idempotent clone, is formed by idempotent lattice polynomial functions, see e.g. \cite{G1,HP1}. In the case of bounded distributive lattices, they coincide with the Sugeno integrals \cite{Marichal}, which represent an important class of aggregation functions.

It is also well known that the classes of idempotent and intermediate aggregation functions coincide. To made the paper 
self-consistent, below we provide a proof of this elementary fact.  

\begin{lemma}\label{lem1}
An aggregation function $f\colon L^n\to L$ is idempotent if and only if it is intermediate.
\end{lemma}

\begin{proof}
Assume that $f$ is idempotent. Then the monotonicity of $f$ yields  
$$ \bigwedge_{i=1}^n x_i = \textstyle f\big(\bigwedge\limits_{i=1}^n x_i,\dots, \bigwedge\limits_{i=1}^n x_i \big)\leq f(x_1,\dots, x_n) \leq f\big(\bigvee\limits_{i=1}^n x_i,\dots, \bigvee\limits_{i=1}^n x_i \big)=\displaystyle \bigvee_{i=1}^n x_i$$
for each $\mathbf{x}\in L^n$.

Conversely, for an arbitrary $x\in L$, from \eqref{eq1} we obtain
$$ x=\bigwedge_{i=1}^n x \leq f(x,\dots,x) \leq \bigvee_{i=1}^n x =x.$$
\end{proof}

In what follows we use the lattice operations $\bigvee$ and $\bigwedge$ with respect to arbitrary, but finite arity. Obviously, both of them can be obtained by finite composition of the binary lattice operations $\vee$ and $\wedge$, respectively. Note that after their application on $n$-ary functions, the values of the resulting function are obtained component-wise, hence the resulting function represents supremum, infimum respectively, in the lattice of aggregation functions. 

\begin{theorem}\label{thm1}
Let $L$ be a finite lattice, and $T\subseteq \mathsf{Agg}(L)$ be a nonempty set containing the lattice operations $\vee$ and $\wedge$ and  that is closed under composition. 
Then for any $f\in T^n$ and any $n$-tuple $\mathbf{a}\in L^n$ the function
\begin{equation}\label{eq2}
h_{f(\mathbf{a})}^T=\bigvee\{g\in T^n\mid g(\mathbf{a})=f(\mathbf{a})\}
\end{equation}
belongs to the set $T$. Moreover the equality 
\begin{equation}\label{eq3}
f(\mathbf{x})=\bigwedge_{\mathbf{a}\in L^{n}} h_{f(\mathbf{a})}^T(\mathbf{x})
\end{equation}
holds for all $\mathbf{x}\in L^n$.
\end{theorem}

\begin{proof}
The finiteness of the lattice $L$ yields that the set $\{g\in T^n\mid g(\mathbf{a})=f(\mathbf{a})\}$ is finite. Moreover it is non-empty, since $f$ is a member of this set. According to the assumptions on the set $T$, we obtain that the function 
$$h_{f(\mathbf{a})}^T=\bigvee\{g\in T^n\mid g(\mathbf{a})=f(\mathbf{a})\}$$
is a composition of functions from $T$, thus it belongs to $T$ as well.
Clearly $f\in \{g\in T^n\mid g(\mathbf{a})=f(\mathbf{a})\}$, which yields $f\leq h_{f(\mathbf{a})}^T$. From this we obtain $f\leq \bigwedge_{\mathbf{a}\in L^{n}} h_{f(\mathbf{a})}^T$.

Conversely for any $\mathbf{x}\in L^n$ we have
$$ \big( \bigwedge_{\mathbf{a}\in L^n} h_{f(\mathbf{a})}^T\big)(\mathbf{x}) \leq h_{f(\mathbf{x})}^T(\mathbf{x})=f(\mathbf{x}),$$
which completes the proof.
\end{proof}

Observe that the function $h_{f(\mathbf{a})}^T$ is the greatest aggregation function $f$, which belongs to the set $T$ satisfying $f(\mathbf{a})=h_{f(\mathbf{a})}^T(\mathbf{a})$.
From this point of view, Theorem \ref{thm1} provides a general method for generating composition-closed subsets of the set $\mathsf{Agg}(L)$. 
To illustrate this method we briefly describe a generating set for $T=\mathsf{Agg}(L)$, $L$ finite lattice. For more details we refer the reader to the paper \cite{HP}.

\begin{example}
Let $f\in \mathsf{Agg}^{n}(L)$ be an $n$-ary aggregation function. It can be shown that 
$$ h_{f(\mathbf{a})}^{\mathsf{Agg}(L)}(\mathbf{x})= 
\begin{cases}
0, \mbox{ if } \mathbf{x}=(0,\dots,0); \\
f(\mathbf{a}), \mbox{ if } \mathbf{0}<\mathbf{x}\leq \mathbf{a}; \\
1, \mbox{ otherwise}.
\end{cases}
$$
Consequently, it suffices to generate the functions $h_{f(\mathbf{a})}^{\mathsf{Agg}(L)}$ from other simpler ones. For this, one can consider the following system of unary and binary aggregation functions. Given $a\in L$ we put
$$ 
\mu_a(x)=\begin{cases}&0, \text{ if }x\leq a, x\neq 1; \\
             &1, \text{ otherwise;}\\   
\end{cases}
\quad
x\oplus_a y=\begin{cases}&1, \text{ if }x=1,y=1; \\
                         &0, \text{ if }x=0,y=0;\\   
                         &a, \text{ otherwise.}    
              \end{cases}
$$
Then for $\mathbf{a}=(a_1,\dots,a_n)\in L^n$ denoting by $\hat{J}_{\mathbf{a}}=\{1\leq i\leq n: a_i\neq 1\}$ we obtain
$$
h_{f(\mathbf{a})}^{\mathsf{Agg}(L)}(\mathbf{x})= \bigvee_{i\in \hat{J}_{\mathbf{a}}}\mu_{a_i}(x_i)\; \vee\; \bigoplus_{i=1}^{n}{\hspace{-1mm}_{f(\mathbf{a})}}\ x_i
$$ 
for all $\mathbf{x}\in L^n$. Note that the symbol $\bigoplus_{i=1}^{n}{\hspace{-1mm}_{f(\mathbf{a})}}$ denotes the composition of the $n-1$ binary functions $\oplus_{f(\mathbf{a})}$. 

Hence, with respect to Theorem \ref{thm1}, the set $\{\mu_a, \oplus_a: a\in L\}$ together with the lattice operations $\vee$ and $\wedge$ generates the set $\mathsf{Agg}(L)$ of all aggregation functions. 
\end{example}

Through the rest of the paper, $L$ will denote a finite lattice. In what follows, for a given $n$-ary idempotent aggregation function $f$ and for each $\mathbf{a}\in L^n$ we identify the functions $h_{f(\mathbf{a})}^{\mathsf{Id}(L)}$ with idempotent aggregation functions of a certain kind. 

Let $n\geq 1$ be a positive integer. For $\mathbf{a}\in L^n$ and $b\in L$ we define 
\begin{equation}\label{eq4}
\chi_{\mathbf{a},b}(\mathbf{x}):=
\begin{cases}
\displaystyle b\wedge \bigvee_{i=1}^n x_i,\ \mbox{if}\; \mathbf{x}\leq\mathbf{a}; \\
\displaystyle \bigvee_{i=1}^n x_i,\ \mbox{in other cases}.
\end{cases}
\end{equation}


\begin{lemma}\label{lem2}
If $\mathbf{a}=(a_1,\dots,a_n)\in L^n$ and $b\in L$ satisfy $\bigwedge_{i=1}^n a_i \leq b$, then $\chi_{\mathbf{a},b}\in \mathsf{Id}^n(L)$.
\end{lemma}

\begin{proof}
First we show that $\chi_{\mathbf{a},b}$ fulfills the boundary conditions and it is monotone. Clearly $\chi_{\mathbf{a},b}(\mathbf{0})=b\wedge \bigvee_{i=1}^n 0=0$. If $\mathbf{a}\neq\mathbf{1}$, we obtain $\chi_{\mathbf{a},b}(\mathbf{1})=\bigvee_{i=1}^n 1=1$. For $\mathbf{a}=\mathbf{1}$ the assumptions yield $1=\bigwedge_{i=1}^n a_i \leq b\leq 1$, hence
$\chi_{\mathbf{a},b}(\mathbf{1})=b \wedge \bigvee_{i=1}^n 1 =b=1$. Further, to prove the monotonicity, let $\mathbf{x},\mathbf{y}\in L^n$ be such that $\mathbf{x}\leq \mathbf{y}$.
If $\mathbf{x}\leq \mathbf{y}\leq \mathbf{a}$, we obtain $\chi_{\mathbf{a},b}(\mathbf{x})=b \wedge \bigvee_{i=1}^n x_i \leq b \wedge \bigvee_{i=1}^n y_i=\chi_{\mathbf{a},b}(\mathbf{y})$. For $\mathbf{y}\nleq \mathbf{a}$ it follows that $\chi_{\mathbf{a},b}(\mathbf{x})\leq \bigvee_{i=1}^n x_i \leq \bigvee_{i=1}^n y_i=\chi_{\mathbf{a},b}(\mathbf{y})$.

To prove idempotency, assume that $x\in L$ is such that $(x,\dots,x)\leq \mathbf{a}$. Consequently $x\leq \bigwedge_{i=1}^n a_i\leq b$, which yields $\chi_{\mathbf{a},b}(x,\dots,x)= b \wedge \bigvee_{i=1}^n x=b\wedge x=x$. For $(x,\dots,x)\nleq \mathbf{a}$ we obtain $\chi_{\mathbf{a},b}(x,\dots,x)=\bigvee_{i=1}^n x=x$.
\end{proof}

%

\begin{lemma}\label{lem3}
Let $f\colon L^n\to L$ be an idempotent aggregation function and $\mathbf{a}\in L^n$. Then $h_{f(\mathbf{a})}^{\mathsf{Id}(L)}=\chi_{\mathbf{a},f(\mathbf{a})}$.
\end{lemma}

\begin{proof}
According to Lemma \ref{lem1}, $\bigwedge_{i=1}^n a_i\leq f(\mathbf{a})$, hence applying Lemma \ref{lem2} the function $\chi_{\mathbf{a},f(\mathbf{a})}$ is idempotent. Further, $\chi_{\mathbf{a},f(\mathbf{a})}(\mathbf{a})= f(\mathbf{a})\wedge \bigvee_{i=1}^n a_i= f(\mathbf{a})$, where the last equality follows from \eqref{eq1} of Lemma \ref{lem1}.
Consequently $\chi_{\mathbf{a},f(\mathbf{a})}\in \{g\in \mathsf{Id}^n(L)\mid g(\mathbf{a})=f(\mathbf{a})\}$ and we obtain 
$$\chi_{\mathbf{a},f(\mathbf{a})}\leq \bigvee\{g\in \mathsf{Id}^n(L)\mid g(\mathbf{a})=f(\mathbf{a})\}=h_{f(\mathbf{a})}^{\mathsf{Id}(L)}.$$

Conversely, $h_{f(\mathbf{a})}^{\mathsf{Id}(L)}(\mathbf{x})\leq \bigvee_{i=1}^n x_i$ for all $\mathbf{x}\in L^n$, since by Theorem \ref{thm1} the function $h_{f(\mathbf{a})}^{\mathsf{Id}(L)}$ is idempotent. As $h_{f(\mathbf{a})}^{\mathsf{Id}(L)}(\mathbf{a})=f(\mathbf{a})$, it follows that $h_{f(\mathbf{a})}^{\mathsf{Id}(L)}(\mathbf{x})\leq f(\mathbf{a})$ for all $\mathbf{x}\in L^n$ satisfying $\mathbf{x}\leq \mathbf{a}$. In addition, we obtain 
$$ h_{f(\mathbf{a})}^{\mathsf{Id}(L)}(\mathbf{x})\leq \bigvee_{i=1}^n x_i, \mbox{ for } \mathbf{x}\nleq \mathbf{a}, $$ 
and
$$ h_{f(\mathbf{a})}^{\mathsf{Id}(L)}(\mathbf{x})\leq f(\mathbf{a}) \wedge \bigvee_{i=1}^n x_i, \mbox{ for } \mathbf{x}\leq \mathbf{a},$$ 
which shows $h_{f(\mathbf{a})}^{\mathsf{Id}(L)} \leq \chi_{\mathbf{a},f(\mathbf{a})}$.

\end{proof}

For generating the idempotent clone $\mathsf{Id}(L)$ we use certain ternary functions. Given $a,b,c,d\in L$ with $a\leq b,d \leq c$, we put

\begin{equation}\label{eq5}
\iota_{(a,b,c),d}(x_1,x_2,x_3):=
\begin{cases}
\displaystyle d\wedge \bigvee_{i=1}^3 x_i,\ \mbox{if}\; (x_1,x_2,x_3)\leq(a,b,c); \\
\displaystyle \bigvee_{i=1}^3 x_i,\ \mbox{in other cases}.
\end{cases}
\end{equation}

Observe that the above defined functions $\iota_{(a,b,c),d}(x_1,x_2,x_3)$ are special instances of the ternary idempotent functions of type $\chi$ defined by \eqref{eq4}. To improve the readability of formulas, we use the following simplified notation. For $\mathbf{a}=(a_1,\dots,a_n)\in L^n$ we sometimes use $\bigwedge \mathbf{a}$ instead of $\bigwedge_{i=1}^n a_i$ and, similarly, $\bigvee \mathbf{a}$ for 
$\bigvee_{i=1}^n a_i$.     

\begin{theorem}\label{thm2}
Let $f\colon L^n\to L$ be an idempotent aggregation function and $\mathbf{a}\in L^n$ be arbitrary. Then 
\begin{equation}\label{eq6}
h_{f(\mathbf{a})}^{\mathsf{Id}(L)}(\mathbf{x})=\bigvee_{i=1}^n \iota_{(\bigwedge \mathbf{a},a_i,\bigvee \mathbf{a}),f(\mathbf{a})}\big(\bigwedge \mathbf{x},x_i, \bigvee \mathbf{x}\big)
\end{equation}
for all $\mathbf{x}\in L^n$.
\end{theorem}

\begin{proof}
Let $\mathbf{a}\in L^n$ be a fixed element. Denote by $\varphi$ the right side of equality \eqref{eq6}. 

By Lemma \ref{lem3}, it suffices to verify that $\varphi$ equals to $\chi_{\mathbf{a},f(\mathbf{a})}$.
Obviously, for any $i\in\{1,\dots,n\}$, $\bigwedge\mathbf{a}\leq a_i \leq \bigvee\mathbf{a}$ holds. Further, 
$\bigwedge\mathbf{a}\leq f(\mathbf{a})\leq \bigvee\mathbf{a}$ follows from Lemma \ref{lem1}. Consequently, the functions $\iota_{(\bigwedge \mathbf{a},a_i,\bigvee \mathbf{a}),f(\mathbf{a})}$ for $i\in\{1,\dots,n\}$ belong to the set of functions defined by \eqref{eq5}. These are idempotent by definition, and the same holds for $\varphi$ as it is a composition of idempotent functions, namely the lattice operations and $\iota$-type functions.

Let $\mathbf{x}\in L^n$ be such that $\mathbf{x}\leq \mathbf{a}$. Then 
$$ \big( \bigwedge_{j=1}^n x_j, x_i,\bigvee_{j=1}^n x_j \big) \leq \big( \bigwedge_{j=1}^n a_j, a_i,\bigvee_{j=1}^n a_j\big)$$ 
for all $i\in\{1,\dots,n\}$, which with respect to \eqref{eq5} yields
$$ \iota_{(\bigwedge \mathbf{a},a_i,\bigvee \mathbf{a}),f(\mathbf{a})}\big(\bigwedge \mathbf{x},x_i, \bigvee \mathbf{x}\big) = f(\mathbf{a})\wedge \big( \bigwedge \mathbf{x}\vee x_i \vee \bigvee \mathbf{x}\big)= f(\mathbf{a})\wedge \bigvee \mathbf{x}$$
for all $i\in\{1,\dots,n\}$. 
Consequently, we obtain
$$ \bigvee_{i=1}^n \iota_{(\bigwedge \mathbf{a},a_i,\bigvee \mathbf{a}),f(\mathbf{a})}\big(\bigwedge \mathbf{x},x_i, \bigvee \mathbf{x}\big) = f(\mathbf{a})\wedge \bigvee \mathbf{x}=\chi_{\mathbf{a},f(\mathbf{a})}(\mathbf{x}).$$

Further, assume $\mathbf{x}\nleq \mathbf{a}$. Lemma \ref{lem1} yields 
$$ \bigvee_{i=1}^n \iota_{(\bigwedge \mathbf{a},a_i,\bigvee \mathbf{a}),f(\mathbf{a})}\big(\bigwedge \mathbf{x},x_i, \bigvee \mathbf{x}\big) \leq \bigvee_{i=1}^n x_i,$$
since the function $\varphi$ is idempotent. On the other hand, due to $\mathbf{x}\nleq \mathbf{a}$, there is an index $j\in \{1,\dots,n\}$ such that $x_j\nleq a_j$. For this particular case we obtain by \eqref{eq5} 
$$ \iota_{(\bigwedge \mathbf{a},a_j,\bigvee \mathbf{a}),f(\mathbf{a})}\big(\bigwedge \mathbf{x},x_j, \bigvee \mathbf{x}\big)= \bigwedge \mathbf{x} \vee x_j \vee \bigvee \mathbf{x} = \bigvee \mathbf{x}, $$
which yields
$$ \bigvee \mathbf{x} \leq \bigvee_{i=1}^n \iota_{(\bigwedge \mathbf{a},a_i,\bigvee \mathbf{a}),f(\mathbf{a})}\big(\bigwedge \mathbf{x},x_i, \bigvee \mathbf{x}\big).$$
Finally, we have shown 
$$ \bigvee_{i=1}^n \iota_{(\bigwedge \mathbf{a},a_i,\bigvee \mathbf{a}),f(\mathbf{a})}\big(\bigwedge \mathbf{x},x_i, \bigvee \mathbf{x}\big) =\bigvee_{i=1}^n x_i=\chi_{\mathbf{a},f(\mathbf{a})}(\mathbf{x}).$$

\end{proof}

Putting Theorems \ref{thm1} and \ref{thm2} together, we obtain the following main theorem of the paper.

\begin{theorem}\label{thm3}
The clone $\mathsf{Id}(L)$ of all idempotent aggregation functions on a finite lattice $L$ is generated by the lattice operations $\vee$ and $\wedge$ and by the ternary $\iota$-type functions 
$\iota_{(a,b,c),d}$ where $a,b,c,d\in L$ with $a\leq b\leq c$ and  $a\leq d\leq c$. Moreover, if $f\colon L^n \to L$ is an idempotent aggregation function, then 
$$ f(\mathbf{x})= \bigwedge_{\mathbf{a}\in L^n} \Big(\bigvee_{i=1}^n \iota_{(\bigwedge \mathbf{a},a_i,\bigvee \mathbf{a}),f(\mathbf{a})}\big(\bigwedge \mathbf{x},x_i, \bigvee \mathbf{x}\big) \Big)$$ 
for all $\mathbf{x}\in L^n$.
\end{theorem}

In what follows we show that the number of generating ternary functions can be significantly reduced.

\begin{lemma}\label{lem4}
Let $a,b,c,d\in L$ be such that $a\leq b\leq c$ and $a\leq d \leq c$. Then
$$ \iota_{(a,b,c),d}(x_1,x_2,x_3)= \iota_{(a,b,1),d}(x_1,x_2,x_3) \vee \iota_{(a,c,1),d}(x_1,x_3,x_3)$$ 
for all $(x_1,x_2,x_3)\in L^3$ satisfying $x_1\leq x_2 \leq x_3$.
\end{lemma}

\begin{proof}
First assume that $(x_1,x_2,x_3)\leq (a,b,c)$ holds. Then $(x_1,x_2,x_3)\leq (a,b,1)$ as well as $(x_1,x_2,x_3)\leq (a,c,1)$, hence we obtain 
$$ \iota_{(a,b,c),d}(x_1,x_2,x_3)= d\wedge \bigvee_{i=1}^3 x_i=d \wedge x_3 = \iota_{(a,b,1),d}(x_1,x_2,x_3) \vee \iota_{(a,c,1),d}(x_1,x_3,x_3).$$

On the other hand, if $(x_1,x_2,x_3)\nleq (a,b,c)$ then $x_1\nleq a$ or $x_2\nleq b$ or $x_3\nleq c$. 
In all three cases it can be easily seen that 
$$ \iota_{(a,b,c),d}(x_1,x_2,x_3)=  x_3 = \iota_{(a,b,1),d}(x_1,x_2,x_3) \vee \iota_{(a,c,1),d}(x_1,x_3,x_3).$$

\end{proof}

\begin{corollary}\label{cor1}
The clone $\mathsf{Id}(L)$ of all idempotent aggregation functions on $L$ is generated by the lattice operations $\vee$ and $\wedge$ and by the ternary $\iota$-type functions
$\iota_{(a,b,1),d}$ where $a,b,d\in L$ with $a\leq b$ and $a\leq d$. 
\end{corollary}

\begin{proof}
Due to Theorem \ref{thm3}, any idempotent aggregation function $f\colon L^n \to L$ can be expressed as 
$$ f(\mathbf{x})= \bigwedge_{\mathbf{a}\in L^n} \Big(\bigvee_{i=1}^n \iota_{(\bigwedge \mathbf{a},a_i,\bigvee \mathbf{a}),f(\mathbf{a})}\big(\bigwedge \mathbf{x},x_i, \bigvee \mathbf{x}\big) \Big).$$
However, given arbitrary $\mathbf{a}\in L^n$ and $i\in \{1,\dots, n\}$, in view of $\bigwedge \mathbf{x} \leq x_i \leq \bigvee \mathbf{x}$ for all $\mathbf{x}\in L^n$, Lemma \ref{lem4} gives
$$ \iota_{(\bigwedge \mathbf{a},a_i,\bigvee \mathbf{a}),f(\mathbf{a})}\big(\bigwedge \mathbf{x},x_i, \bigvee \mathbf{x}\big)= \iota_{(\bigwedge \mathbf{a},a_i,1),f(\mathbf{a})}\big(\bigwedge \mathbf{x},x_i, \bigvee \mathbf{x}\big) \vee \iota_{(\bigwedge \mathbf{a},\bigvee \mathbf{a},1),f(\mathbf{a})}(\bigwedge \mathbf{x},\bigvee \mathbf{x},\bigvee \mathbf{x}). $$
 
\end{proof}

Obviously, the number of generators described in Corollary \ref{cor1} depends on the order structure of a considered lattice. We enumerate the number of these generators in two extremal cases: if $L$ is a chain and in the case when $L\cong M_{n-2}$, i.e., $L$ being isomorphic to a lattice consisting of $n-2$ mutually incomparable elements together with a bottom and a top element respectively. With respect to a fixed finite cardinality, the first case covers lattices with the most number of comparable elements, while the second case with the smallest number of comparable elements.

For a positive integer $n\geq 2$ consider the finite $n$-element chain $\mathbf{n}=\{0,1,\dots,n-1\}$ with the usual order. We enumerate the cardinality of $G_{\mathbf{n}}$, the generating set of $\mathsf{Id}(\mathbf{n})$ described in Corollary \ref{cor1}. In this case $G_{\mathbf{n}}=\big\{ \iota_{(i,j,n-1),k}: i,j,k\in\{0,\dots,n-1\}, i\leq j,k\big\}\cup\{\vee,\wedge\}$. Obviously, for a fixed $i\in \{0,\dots,n-1\}$ there are $(n-i)^2$ possibilities for indexes $j,k$ satisfying $i\leq j,k$. Hence we obtain 
$$ \left|G_{\mathbf{n}}\right|= n^2+ (n-1)^2 + \dots + 1^2 +2 = \sum_{i=1}^n i^2 = \frac{n}{3}(n+1)\left(n+\frac{1}{2}\right)+2,$$
i.e., $\left|G_{\mathbf{n}}\right|=O(n^3)$.

For the lattice $M_{n-2}$, $n\geq 4$ we obtain that the generating set $G_{M_{n-2}}$ consists of the lattice operations and $n^2+(n-2)\cdot 4+1$ functions of the type $\iota$. Hence in this case $\left|G_{M_{n-2}}\right|=n^2+4n-5$ and we can see that in a general case, the number of generators $\left|G_L\right|$, $L$ an $n$-element lattice, varies between $O(n^2)$ and $O(n^3)$.

\section{Concluding remarks}
In this paper we have presented explicitly certain generating sets of all intermediate (idempotent) aggregation functions on finite lattices. We believe that our results will be convenient for further analysis of this class with respect to better understanding how idempotent aggregation functions are constructed. In the future work we expect the extension of our results for certain important composition-closed subclasses of idempotent aggregation functions, especially for the class of Sugeno integrals on bounded distributive lattices.

\section*{Acknowledgements}
The first two authors were supported by the international project Austrian Science Fund (FWF)-Grant Agency of the Czech Republic (GA\v{C}R) number 15-346971; the third author by the Slovak VEGA Grant 1/0420/15; the fourth author by the IGA project of the faculty of Science Palack\'y University Olomouc no. PrF2015010 and by the Slovak VEGA Grant no. 2/0044/16.

\end{document}